 \documentclass[preprint,review,12pt]{elsarticle}

\usepackage{graphicx}
\usepackage{amssymb}
\usepackage{amsthm}
\usepackage{amsmath}
\usepackage{mathrsfs}
\usepackage{bbold}

\usepackage{lineno}

\newtheorem{theorem}{Theorem}[section]
\newtheorem{assumption}{Assumption}

\newtheorem{lemma}{Lemma}[section]

\newproof{pf}{Proof}
\newcommand{\leb}{\text{Leb}}

\biboptions{comma}

\journal{Probability and Statistics Letters}

\begin{document}

\begin{frontmatter}


\title{On the stability of the stochastic gradient Langevin algorithm with dependent data stream\tnoteref{lendulet}}

\tnotetext[lendulet]{Both authors were supported
by the ``Lend\"ulet'' grant 2015-6 of the Hungarian
Academy of Sciences.}

\author{Mikl\'os R\'asonyi\fnref{label2}}
\fntext[label2]{Alfr\'ed R\'enyi Institute of Mathematics, Re\'altandoda utca 13-15, 1053 Budapest, Hungary}

\author{Kinga Tikosi\fnref{label2,label3}}
\fntext[label3]{During the preparation of this paper the author attended the PhD school of Central European University, Budapest.}


\begin{abstract}
We prove, under mild conditions, that 
the stochastic gradient Langevin dynamics
converges to a limiting law as time tends to infinity,
even in the case where the driving data sequence
is dependent.
\end{abstract}


\end{frontmatter}

%\linenumbers


\section{Stochastic gradient Langevin dynamics}

Sampling from high-dimensional, possibly not even logconcave distributions is
a challenging task, with far-reaching applications in optimization, in particular,
in machine learning, see \citet{raginsky,5,6,brosse2018promises}.
 
Let $U:\mathbb{R}^{d}\to \mathbb{R}_{+}$ be a given function and consider
the corresponding Langevin equation
\begin{equation}\label{landi}
d\Theta_{t}=-\nabla U(\Theta_{t})\, dt+\sqrt{2}\, dW_{t},
\end{equation}
where $W$ is a $d$-dimensional standard Brownian motion.
Under suitable assumptions, the unique invariant probability $\mu$
for the diffusion process \eqref{landi} has a density (with respect
to the $d$-dimensional Lebesgue measure) that is proportional
to $\exp(-U(x))$, $x\in\mathbb{R}^{d}$.

In practice, Euler approximations of \eqref{landi} may be used for sampling
from $\mu$, i.e.\ a recursive scheme 
\begin{equation}\label{labi}
\vartheta_{t+1}^{\lambda}=\vartheta_{t}^{\lambda}-\lambda
\nabla U(\vartheta^{\lambda}_{t})+
\sqrt{2\lambda}\xi_{t+1}
\end{equation}
is considered for some small $\lambda>0$ and independent standard $d$-dimensional Gaussian
sequence $\xi_{i}$, $i\geq 1$.

In some important applications, however, $U,\nabla U$ are unknown, one disposes
only of unbiased estimates $H(\theta,Y_{t})$, $t\in\mathbb{N}$ of $\nabla U(\theta)$,{}
where $Y_{t}$ is some stationary data sequence.  
From this point on we switch to rigorous mathematics.

Let us fix integers $d,m\geq 1$ and a probability space $(\Omega,\mathcal{F},\mathbb{P})$. 
$\mathscr{B}(\mathscr{X})$ denotes the $\sigma$-algebra of the Borel-sets of a Polish space $\mathscr{X}$. For a random variable $X$, $\mathscr{L}(X)$ denotes its law.
The Euclidean norm on $\mathbb{R}^d$ or $\mathbb{R}^m$ will be denoted by $|\cdot|$, while $||\cdot||_{TV}$ stands for 
the total variation distance of probability measures
on $\mathcal{B}(\mathbb{R}^d)$. Let $B_r:=\{\theta\in\mathbb{R}^k:\,
|\theta|\leq r\}$ denote the ball of radius $r$, for
$r\geq 0$, for both $k=d$ and $k=m$, depending
on the context. The notation $\leb(\cdot)$ refers to the $d$-dimensional Lebesque-measure.
%and for $p\geq 1$ $L^p$ is the space of $p$-integrable random variables.  

For $0<\lambda\leq 1$, $t=0,1,\dots$ and for a constant initial value $\theta_0\in\mathbb{R}^d$ consider the recursion
\begin{equation}\label{alg:SGLD}
    \theta_{t+1}^\lambda=\theta_{t}^\lambda-\lambda H(\theta_t^\lambda,Y_t)+\sqrt{\lambda}\xi_{t+1},
    \ t\in\mathbb{N},\ \theta_0^\lambda:=\theta_0,
\end{equation}
where $\xi_i$, $i\geq 1$ is an i.i.d.\ sequence of $d$-dimensional random variables with independent coordinates such that $\mathbb{E}[\xi_{i}]=0$ and $E[|\xi_{i}|^2]=
\sigma^2$ for some $\sigma^2$.
Furthermore, the density function $f$ of $\xi_i$
with respect to $\leb$
is assumed strictly positive on every compact set.
Assume that $(Y_t)_{t\in\mathbb{Z}}$ is a strict sense stationary process with values in $\mathbb{R}^m$ an it is independent of the noise process $(\xi_t)_{t\geq 1}$.
Finally, $H:\mathbb{R}^d\times\mathbb{R}^m\to\mathbb{R}^d$
is a measurable function.

%Let $\theta\in\mathbb{R}^d$ be a parameter and $X=\{x_i, i\in\{1,2,\dots N\}\}$ a set of data points with a large $N.$ $\theta$ has a prior distribution $p(\theta).$

%Then SGLD is meant to sample from the posterior distribution $p(\theta|x)\propto p (\theta) \prod_{i = 1}^N p(x_i \vert \theta)$ using the recursion 
% \begin{align} \label{alg:SGLD_minibatch}
%      \theta_{t+1} = \theta_t- \frac{\lambda_t} 2 \left( \nabla \log p(\theta_t) + \frac{N}{n} \sum_{i=1}^n \nabla \log p(x_{t_i} \vert \theta_t) \right) + \sqrt{\lambda_t}\xi_{t+1},
% \end{align}
%where $p(x_{t_i} \vert \theta_t)$ is the likelihood of the data conditioned on $\theta_t$ and $\xi_{t+1}$ is some noise with $\mathbb{E}[\xi_{t+1}]=0.$ 

%Algorithm (\ref{alg:SGLD_minibatch}) can be interpreted as a minibatch version of the Euler-discretization of the Langevin stochastic differential equation:
%$$ \theta_{t+1} = \theta_t- \frac{\lambda_t}{2}  \nabla U(\theta_t)  + \sqrt{\lambda_t}\xi_{t+1},$$
%where $U(\theta)=\sum_{i=1}^N-\log(p(x_i\vert\theta))$. 

A particular case of \eqref{alg:SGLD} is the
\emph{stochastic gradient Langevin dynamics} (SGLD), introduced in \citet{welling2011bayesian}, 
designed to learn from large datasets. 
See more about different versions of SGLD and their 
connections in \citet{brosse2018promises}. 
Note that in the present setting, unlike in SGLD, we do not assume that $H$ is the gradient of a function and we do not assume
$\xi_i$ to be Gaussian. 

A setting similar to ours was considered in \citet{lovas2019markov} under different assumptions. We will compare our results
to those of \citet{lovas2019markov} at the end of Section
\ref{S:1} below.

The sampling error of $\theta_{t}^{\lambda}$ has been thoroughly analysed in
the literature: $d(\theta_{t}^{\lambda},\mu)$ has been estimated for various probability metrics $d$,
see \citet{5,6,raginsky,brosse2018promises}. The ergodic behaviour of $\theta_{t}^\lambda$,
however, has eluded attention so far. If $Y_{t}$ are i.i.d.\ then $\theta_{t}^{\lambda}$
is a homogeneous Markov chain and standard results of Markov chain theory apply.
In the more general, stationary case (considered in \citet{6,5}), however,
that machinery is not available. In the present note we study scheme \eqref{alg:SGLD} with stationary $Y_t$
and establish that its law converges to a limit
in total variation.

\section{Main results}
\label{S:1}

\begin{assumption} \label{asp:dissi}
There is a constant $\Delta> 0$ and a measurable
function $b:\mathbb{R}^m\to\mathbb{R}_+$ such that, for all $\theta\in\mathbb{R}^d$ and $y\in\mathbb{R}^m$ 
\begin{equation}\label{test}
\left<H(\theta, y),\theta\right>\geq\Delta|\theta|^2-b(y).
\end{equation}
\end{assumption}

\begin{assumption}\label{asp:growth_of_H}
There exist constants $K_1,K_2,K_3>0$ and $\beta\geq 1$ such that 
\begin{equation}
    |H(\theta,y)|\leq K_1 |\theta| + K_2|y|^{\beta}+K_3.
\end{equation}
\end{assumption}

\begin{assumption}\label{asp:moments} There exist (finite)
constants $M_y,M_b>0$ such that $\mathbb{E}[|Y_0|^{2\beta}]\leq M_y$ and $\mathbb{E}[b(Y_0)]\leq M_b$.
\end{assumption}

\begin{theorem}\label{main}
Let Assumptions \ref{asp:dissi}, \ref{asp:growth_of_H} and \ref{asp:moments} hold. Then, for $\lambda$ small enough,  
the law $\mathscr{L}(\theta_t^\lambda)$ of the iteration defined by (\ref{alg:SGLD}) converges in total variation as $t\rightarrow\infty$ and the limit does not depend on
the initialization $X_0$.
\end{theorem}

In \citet{lovas2019markov}, $\Delta$ in \eqref{test}
was allowed to depend on $y$ but $b$ in \eqref{test}
had to be constant, the process $Y$ was assumed
bounded and the process $\xi$ Gaussian.
Furthermore, in Assumption \ref{asp:growth_of_H}, 
$\beta$ had to be $1$. Under these conditions the
conclusion of Theorem \ref{main} was obtained, together
with a rate estimate.

Theorem \ref{main} above
complements the results of \citet{lovas2019markov}:
$\Delta$ must be constant in our setting but the restrictive
boundedness hypothesis on $Y$ could be removed,
$\xi$ need not be Gaussian, $\beta$ in 
\ref{test} can be arbitrary
and $b$ in \ref{test} may 
depend on $y$. The examples in Section \ref{sec:ex} demonstrate that our present results 
cover a wide range of relevant applications
where the obtained generalizations are crucial.

\section{Markov chains in random environment}

The rather abstract Theorem \ref{totoro} below,
taken from \citet{gerencser2020invariant}, is 
the key result we use in this paper. Let us first recall the related terminology and the assumptions.

Let $\mathscr{X}$ and $\mathscr{Y}$ be Polish spaces and let $(\mathscr{X}_n)_{n\in\mathbb{N}}$  
(resp.\ $(\mathscr{Y}_n)_{n\in\mathbb{N}}$)
be a non-decreasing sequence of (non-empty) Borel-sets in $\mathscr{X}$ (resp.\ $\mathscr{Y}$).
Consider a parametric family of transition kernels, i.e. a map $Q:\mathscr{Y}\times\mathscr{X}\times\mathscr{B}(\mathscr{X})\rightarrow [0,1]$ such that for all $B\in\mathscr{B}(\mathscr{X})$ the function $(x,y)\rightarrow Q(x,y,B)$ is measurable and for every $(x,y)\in\mathscr{X}\times \mathscr{Y}$ $Q(x,y,\cdot)$ is a probability.

An $\mathscr{X}$ valued stochastic process $(X_t)_{t\in\mathbb{N}}$ is called a \emph{Markov chain in a random environment} with transition kernel 
$Q$ if $X_0\in\mathscr{X}$ is deterministic (for simplicity) and 
\begin{equation}
 \mathbb{P}(X_{t+1}\in A|\mathcal{F}_t)=Q(X_t,Y_t,A) \text{, for } t\in\mathbb{N},   
\end{equation}
 where we use the filtration $\mathcal{F}_t=\sigma(Y_k, k\in\mathbb{Z}; X_j, 0\leq j\leq t)$.

For a parametric family of transition kernels $Q$ and a bounded (or non-negative) function $V:\mathscr{X}\rightarrow \mathbb{R}$ define
\begin{equation}
    [Q(y)V](x)=\int_{\mathscr{X}}V(z)Q(x,y,dz)\text{, for }x\in\mathscr{X}.
\end{equation}

\begin{assumption}\label{asp:tightness} Let the process $(X_t)_{t\in\mathbb{N}}$ started from $X_0$ with 
be such that 
\begin{equation}
    \sup_{t\in\mathbb{N}}\mathbb{P}(X_t \notin\mathscr{X}_n)\rightarrow 0, n\rightarrow \infty.
\end{equation}
\end{assumption}

%Let $\mathfrak{M}$ denote the set of probability laws $\mu$ on $\mathscr{X}\times \mathscr{Y}^{\mathbb{Z}}$ such that %the second marginal has the same law as $(Y_k)_{k\in\Z}$ and for which% 
%the process $X_t$ started from $X_0$ with the law $\mathscr{L}(X_0,(Y_k)_{k\in\mathbb{Z}})=\mu$ satisfies Assumption \ref{asp:tightness}.

\begin{assumption}\label{asp:minorization} (Minorization condition) Let $\mathbb{P}(Y_0\notin\mathscr{Y}_n), n\rightarrow\infty.$ Assume that there exists a sequence of probability measures $(\nu_n)_{n\in\mathbb{N}}$ and a non-decreasing sequence $(\alpha_n)_{n\in\mathbb{N}}$ with $\alpha_n\in(0,1]$ such that  for all $n\in\mathbb{N}, x\in\mathscr{X}_n, y\in\mathscr{Y}_n$, and $A\in\mathcal{B}(\mathcal{X}),$ \begin{equation}
    Q(x,y,A)\geq\alpha_n\nu_n(A).
\end{equation}
\end{assumption}

\begin{theorem}\label{totoro} (Theorem 2.11. of \citet{gerencser2020invariant}) Let Assumptions \ref{asp:tightness} and \ref{asp:minorization} hold. Then there exists a probability $\mu_*$ on $\mathscr{B}(\mathscr{X}\times \mathscr{Y}^{\mathbb{Z}})$ such that $$||\mathscr{L}(X_t,(Y_{t+k})_{k\in\mathbb{Z}})-\mu_*||_{TV}\rightarrow0,\textit{ as } t\rightarrow\infty.$$ If $(X'_t)_{k\in\mathbb{N}}$ is another such Markov chain 
started from a different $X'_0$ satisfying Assumption
\ref{asp:tightness} then $$||\mathscr{L}(X_t,(Y_{t+k})_{k\in\mathbb{Z}})-
\mathscr{L}(X_t',(Y_{t+k})_{k\in\mathbb{Z}})||_{TV}\rightarrow 0,\textit{ as } t\rightarrow\infty.\quad \Box$$ 
\end{theorem}

\section{Proofs}

Define the  Markov chain associated to the recursive scheme (\ref{alg:SGLD}) as \begin{equation}
    Q(\theta,y,A)=\mathbb{P}(\theta-\lambda H(\theta,y)+\sqrt{\lambda}\xi_{n+1}\in A),
\end{equation} for all $y\in\mathscr{Y}:=\mathbb{R}^m$, $\theta\in\mathscr{X}:=\mathbb{R}^d$ and $A\in\mathscr{B}(\mathbb{R}^d)$. 

\begin{lemma}\label{obu1} For small enough $\lambda$,
under Assumptions \ref{asp:dissi} and \ref{asp:growth_of_H}, the process $(\theta_t^\lambda)_{t\in\mathbb{N}}$ given by recursion (\ref{alg:SGLD}) satisfies Assumption \ref{asp:tightness} with $\mathscr{X}_n:=B_n$ (the ball
of radius $n$).
\end{lemma}
\begin{pf}
Choose $V(\theta)=|\theta|^2$. Then, since $E\xi_1=0$,
\begin{align*}
    [Q(y)V](\theta)&=\mathbb{E}[V(\theta-\lambda H(\theta,y)+\sqrt{\lambda}\xi_1)]\\
    &=|\theta|^2+\lambda^2|H(\theta,y)|^2+\lambda\mathbb{E} |\xi_1|^2-2\lambda\left<\theta,H(\theta,y)\right> \\
    &\leq (1-2\lambda\Delta)|\theta|^2+\lambda(\sigma^2 +2b(y))+3\lambda^2(K_1^2|\theta|^2+K_2^2|y|^{2\beta}+K_3^2) \\
    &=(1-2\lambda\Delta +3\lambda^2 K_1^2)V(\theta)+\lambda(\sigma^2 +2b(y))+3\lambda^2(K_2^2|y|^{2\beta}+K_3^2)\\
    &=\gamma V(\theta)+K(y),
\end{align*}
with $K(y)=\lambda(\sigma^2 +2b(y))+3\lambda^2[K_2^2|y|^{2\beta}+K_3^2]$ and $\gamma=(1-2\lambda\Delta +3\lambda^2 K_1^2)$. Note that for small enough $\lambda$,  $\gamma\in(0,1)$, independent of $y$. 

Now using Lemma \ref{lem:multipleQ} below and setting $\theta=\theta_0$ and $y_k=Y_k$ for $k\geq 1$
we get, for each $t\geq 1$,
\begin{align*}
     \mathbb{E}|\theta_t^\lambda|^2 &=\mathbb{E} [Q(Y_t)Q(Y_{t-1})\dots Q(Y_1)V](\theta_0)
     \leq \gamma^t V(\theta_0)+\sum_{i=1}^t \gamma^{i}\mathbb{E} K(Y_i) \\
     &= \gamma^t |\theta_0|^2+\sum_{i=1}^t \gamma^{i} [\lambda(\sigma^2 +2\mathbb{E}[b(Y_i)])+3\lambda^2(K_2^2\mathbb{E}|Y_i|^{2\beta}+K_3^2)]\\
     &\leq |\theta_0|^2+\frac{\gamma}{1-\gamma} [(\sigma^2 +2M_b)+3(K_2^2 M_y+K_3^2)]
     <\infty, 
\end{align*}
by Assumption \ref{asp:moments}. Then, using Markov's inequality, we arrive at
\begin{equation}
     \mathbb{P}(\theta_t^\lambda\notin\mathscr{X}_n)= \mathbb{P}(|\theta_t^\lambda|>n)\leq\frac{\sup_t
     \mathbb{E}
     |\theta_t^\lambda|^2}{n^2}\rightarrow 0\text{, as } n\rightarrow\infty.\quad \Box
\end{equation}
\end{pf}

\begin{lemma}\label{lem:multipleQ}
Assume $[Q(y)V](\theta) \leq \gamma V(\theta)+K(y)$. Then
\begin{equation} 
     [Q(y_k)Q(y_{k-1})\dots Q(y_1)V](\theta)\leq \gamma^k V(\theta)+\sum_{i=1}^k \gamma^{i-1}K(y_i).
\end{equation}
\end{lemma}
\begin{pf}
We prove the statement by induction. For $k=1$, it is
true by assumption. Using that \begin{equation}
    [Q(y_2)Q(y_1)V](x)=\int_{\mathscr{X}}Q(x,y_2,dr)\int_{\mathscr{X}}V(z)Q(r,y_1,dz)\text{, for }r\in\mathscr{X},
\end{equation}    
for $k>1$ we get
\begin{align*} 
     [Q(y_k)Q(y_{k-1})\dots Q(y_1)V](\theta)&= \int_{\mathscr{X}} Q(\theta,y_k,dx) [Q(y_{k-1})Q(y_{k-2})\dots Q(y_1)V](x)\\
     &\leq \int_{\mathscr{X}} \left( \gamma^{k-1} V(x)+\sum_{i=1}^{k-1} \gamma^{i-1}K(y_i)\right) Q(\theta,y_k,dx)\\
     &=\gamma^{k-1} \int_{\mathscr{X}}V(x)Q(\theta,y_k,dx) + \sum_{i=1}^{k-1} \gamma^{i-1}K(y_i)  \\
     &\leq \gamma^k V(\theta)+\sum_{i=1}^k \gamma^{i-1}K(y_i).
\quad \Box\end{align*}
\end{pf}

\begin{lemma}\label{obu2}
Define $\mathscr{X}_n=B_n$, $\mathscr{Y}_n:=B_n$,
$n\in\mathbb{N}$ and let Assumptions \ref{asp:dissi} and \ref{asp:growth_of_H} hold. Then Assumption \ref{asp:minorization} is satisfied, for all $\lambda$. 
\end{lemma}
\begin{pf}For all $A\in\mathscr{B}(\mathscr{X})$,
\begin{align*}
    Q(\theta,y,A)&=\mathbb{P}(\theta-\lambda H(\theta,y)+\sqrt{\lambda}\xi_{1}\in A) \\ &\geq \int_{\mathbb{R}^d} \mathbb{1}_{\{\theta-\lambda H(\theta,y)+\sqrt{\lambda}\xi_{1}\in A\cap B_n\}}f(w)dw \\
    &= \frac{1}{\lambda^{d/2}}\int_{A\cap B_n} f\left(\frac{z-\theta+\lambda H(\theta,y)}{\sqrt{\lambda}}\right)dz \\ &\geq  \frac{\leb(A\cap B_n)}{\lambda^{d/2}} C(n) =\frac{\leb(A\cap B_n)}{\leb(B_n)}\frac{C(n) \leb(B_n)}{\lambda^{d/2}},
\end{align*}
where we use that for $\theta,z\in B_n$ and $y\in B_{n} $ we have $$\left| \frac{z-\theta+\lambda H(\theta,y)}{\sqrt{\lambda}}\right|\leq \frac{n+n+\lambda(K_1n+K_2n^{\beta}+K_3)}{\sqrt{\lambda}}=:R(n),$$
therefore the integrand can be bounded from below by $C(n):=\inf_{x\in B_{R(n)}} f(x)>0$. Then define $  \nu_n(A):=\frac{\leb(A\cap B_n)}{\leb(B_n)}$ and $\alpha_n:=\frac{C(n) \leb(B_n)}{\lambda^{d/2}},$ which proves that Assumption \ref{asp:minorization} holds.\hfill $\Box$
\end{pf}

\begin{pf}[of Theorem \ref{main}.]
Follows from Lemmas \ref{obu1}, \ref{obu2} and Theorem
\ref{totoro}.\hfill $\Box$
\end{pf}

\section{Examples}\label{sec:ex}

\subsection{Nonlinear regression}

Let us consider a nonlinear regression problem which can also be seen as a one layer neural network in a supervised learning setting, where only one trainable layer connects the input and the output vectors. The training set consists of entries $Y_t=(Z_t,L_t)$ with the features $Z_t\in\mathbb{R}^{d_0}$ and the corresponding labels $L_t\in\mathbb{R}^{d_1}$ for $t\in 1,\dots,N.$ We assume that
$Y_t$ is a stationary process. Set $m:=d_0+d_1$, the dimension
of $Y_t$.

The trainable parameters will be a
matrix $W\in\mathbb{R}^{d_0\times d_1}$ and a vector
$g\in\mathbb{R}^{d_1}$, therefore the dimension of
$\theta:=(W,g)$ will be $d=d_0d_1+d_1$.
The prediction function 
$h:\mathbb{R}^{d_0}\times\mathbb{R}^{d}\to \mathbb{R}^{d_1}$ is defined by $h(z,\theta):=s(Wz+g),$ where 
$s=(s_1,\ldots,s_{d_1})$ is a collection of nonlinear activation
functions $s_i:\mathbb{R}\rightarrow \mathbb{R}$ for $i=1,\dots,d_1$.
We will assume that each $s_i$ and their derivatives
$s_{i}'$ are all bounded by some constant $M_s$ for $i=1,\dots,d_1$.

Choosing the loss function to be mean-square error, one aims to minimize the empirical risk, that is  
\begin{equation}
\min \{\mathbb{E}[|h(Z_t,\theta)-L_t|^2]+\kappa|\theta|^2\},
\end{equation}
with some $\kappa>0$, where the 
second term is added for regularization. 

It is standard to solve this optimization step using gradient-based methods. For $y=(z,l)\in\mathbb{R}^{d_0}\times\mathbb{R}^{d_1}$ denote $U(\theta,y)=|h(z,\theta)-l|^2+\kappa|\theta|^2$ and the updating function to be used in the algorithm will be 
\begin{equation}
   H(\theta,y)=\nabla U(\theta,y)=\frac{\partial}{\partial\theta}
|h(z,\theta)-l|^2+2\kappa\theta. 
\end{equation}

\begin{lemma}
The function $H(\theta, y)$ defined as above satisfies Assumptions \ref{asp:dissi} and \ref{asp:growth_of_H}.
\end{lemma}
\begin{proof} 
Using the chain rule, a short calculation
gives
%to get the gradient of $|h(z,\theta)-l|^2$,  \begin{align*}\frac{\partial}{\partial\theta}|h(z;\theta)-l|^2&=\frac{\partial}{\partial\theta}\left(\sum_{j=1}^{d_1}(s_j(W_jz+g_j)-l_j)^2)\right)\\
%  &=  \begin{bmatrix}
%           2 (h(z;\theta)_1-l_1){s_1}'(W_1z+g_1)z_1\\
%           \vdots \\
%           2 (h(z;\theta)_1-l_1){s_1}'(W_1z+g_1)z_{d_0} \\
%           2 (h(z;\theta)_2-l_2){s_2}'(W_2z+g_2)z_1 \\
%           \vdots \\
%           2 (h(z;\theta)_{d_1}-l_{d_1}){s_{d_1}}'(W_{d_1}z+g_{d_1})z_{d_0} \\
%           2 (h(z;\theta)_1-l_1){s_1}'(W_1z+g_1)  \\
%           \vdots \\
%           2 (h(z;\theta)_{d_1}-l_{d_1}){s_{d_1}}'(W_{d_1}z+g_{d_1}) . \\
%           \end{bmatrix}\in\mathbb{R}^{d_1(d_0+1)}
%  \end{align*}
 \begin{equation}
    \left|\frac{\partial}{\partial\theta}|h(z;\theta)-l|^2\right| = \sqrt{\sum_{i=1}^{d_0+1}\sum_{j=1}^{d_1}\left( 2 (h(z;\theta)_j-l_j){s_j}'(\left<W_j,z\right>+g_j)z_{i}\right)^2 },
\end{equation}
where we define $z_{d_0+1}=1$ and $W_j$ stands for the $j$th row of $W$.  Notice that by the boundedness of $s'$ and $s$ this is bounded in $\theta$ and at most quadratic in $y$. Then Assumption \ref{asp:growth_of_H} is satisfied with $\beta=2$.

Using the same argument about the boundedness of $s$ and $s'$
\begin{align*}
    \left|\left<\frac{\partial}{\partial\theta}|h(z;\theta)-l|^2,\theta\right>\right| &= \left| \sum_{i=1}^{d_0+1}\sum_{j=1}^{d_1} 2 (h(z;\theta)_j-l_j){s_j}'(W_jz+g_j)z_{i}\theta_{i,j} \right| \\ 
    &\leq C_0 d M_s \left(|y|^2+1\right)\left|\theta\right|,
\end{align*} for some $C_0>0.$
%Note that $|y|\geq|z|$ trivially holds. 
Using that $\left<\frac{\partial}{\partial\theta}\kappa|\theta|^2,\theta\right>=2\kappa|\theta|^2,$ we get that
$\left<\nabla U(\theta),\theta\right> \geq c|\theta|^2 - C(|y|^4+1)$ with some $c,C$ therefore Assumption \ref{asp:dissi} is satisfied with $b(y)$ being of degree 4 in $y$.
\end{proof}

%The most widely used differentiable activation functions have bounded derivatives (sigmoid: $\sigma(x)=\frac{1}{1 + e^{-x}}$, $\tanh(x) = \frac{e^x - e^{-x}}{e^x + e^{-x}}$, Gaussian linear unit: $x\Phi(x)$, softsign $\frac{x}{1+|x|}$, etc.)

\subsection{A tamed algorithm for neural networks}

It has been observed that in multi-layer neural networks quadratic regularization is not sufficient to guarantee dissipativity, while adding a higher order term would violate Lipschitz continuity. So the standard SGLD algorithm diverges anyway.
To remedy this, certain ``tamed'' schemes have been suggested in
\citet{lovas2020taming}.

In contrast to the previous case now we will hidden layers between the input and output: layer $0$ is the input, layer $n$ is the output and $1,\dots,n-1$ are the hidden layers of the neural network for some $n>1$. The prediction function $h$ will be defined as the composition of a sequence of $n+1$ linear transformations and activation functions, i.e. $h(z,\theta)=s_n(W_ns_{n-1}(W_{n-1}\dots s_0(W_0z)))$ where $\theta$ is the collection of all parameters $W_i\in\mathbb{R}^{d_{i-1}}\times\mathbb{R}^{d_{i}}$, $i=1,\dots,n$ and $s_i:\mathbb{R}^{d_i}\rightarrow\mathbb{R}^{d_i}$ is a componentwise non-linear activation function,
assumed bounded together with its derivatives
by some constant $M_s$. Therefore $h:\mathbb{R}^{d_0}\times\mathbb{R}^d\rightarrow\mathbb{R}^{d_n},$ where $d=\sum_{i=1}^n d_{i-1}d_i$ is the dimension of $\theta$. For the case of simplicity in this case we assumed that there is no bias term $g$. The training set consists of entries $Y_t=(Z_t,L_t)$ with the features $Z_t\in\mathbb{R}^{d_0}$ and the corresponding labels $L_t\in\mathbb{R}^{d_n}$, the dimension of each $Y_t$ is $m=d_0+d_n.$ We assume that
$Y_t$ is a stationary process. 

As in the previous subsection, the 
regularized empirical risk has the form 
$U(\theta,y)=|h(z,\theta)-l|^2+\frac{\eta }{2(r+1)}|\theta|^{2(r+1)}$ with
some $r\geq 0$, $\eta>0$.
Denoting $G(\theta,y)=\nabla U(\theta,y)$, the ``tamed'' updating function we use will be defined as 
$H(\theta, y):=\frac{G(\theta, y)}{1+\sqrt{\lambda}|\theta|^{2 r}},$ for every $\theta \in \mathbb{R}^{d}, y \in \mathbb{R}^{m}$. Note
that this function depends on $\lambda$!

We will use the following.
\begin{lemma}\label{vari}(Proposition 4 of \citet{lovas2020taming}) \label{lem:layers_regularization_connection}
\begin{equation}
    \left\vert \frac{\partial}{\partial\theta} |h(z,\theta)-l|^2 \right\vert\leq C (1+|y|)^{2}\left(1+|\theta|^{n+1}\right), 
\end{equation} where $C>0$ depends on $D=\max_{j=1,\dots,n} d_j$, $n$ and $M_s$.\hfill $\Box$
\end{lemma}

\begin{lemma}\label{triste} For $\lambda$ small enough, the conclusions of Theorem \ref{main} hold
for the scheme \eqref{alg:SGLD}
with $H(\theta, y)$ defined as above, provided that
$r\geq \frac{n+2}{2}$ and Assumption \ref{asp:moments}
holds. 
\end{lemma}
\begin{proof} 
Using Lemma \ref{vari}, Assumption
\ref{asp:growth_of_H} can be checked as follows: 
\begin{align*}
    \left|H(\theta, y)\right|&=\left|\frac{\frac{\partial}{\partial\theta} {|h(z;\theta)-l|^2+\eta \theta|\theta|^{2r}}}{1+\sqrt{\lambda}|\theta|^{2 r}}\right|\\
    &\leq\left|\frac{ C (1+|y|)^{2}\left(1+|\theta|^{n+1}\right)}{1+\sqrt{\lambda}|\theta|^{2 r}}\right|+\left|\frac{\eta \theta|\theta|^{2r}}{1+\sqrt{\lambda}|\theta|^{2 r}}\right|
    \leq K_1 |\theta| + K_2|y|^{\beta}+K_3,
\end{align*} where $K_1=\frac{\eta}{\sqrt{\lambda}}$, $\beta=2$ and the constants $K_2$ and $K_3$ depend on $\lambda, \eta, n$ and $C.$

Let us check Assumption \ref{asp:dissi}. For the regularization term we have 
\begin{equation}\left<\frac{\eta \theta|\theta|^{2r}}{1+\sqrt{\lambda}|\theta|^{2 r}},\theta\right>=\frac{\eta|\theta|^{2r+2}}{1+\sqrt{\lambda}|\theta|^{2 r}}\geq 
\min\left\{\frac{\eta}{2\sqrt{\lambda}},\frac{\eta}{2} \right\}
|\theta|^2\geq  \frac{\eta}{2}|\theta|^2
\end{equation}
for $\lambda$ small enough.
%- \frac{\eta}{\sqrt{\lambda}}\frac{|\theta|^2}{1+\sqrt{\lambda}|\theta|^{2r}}=\frac{\eta}{\sqrt{\lambda}}|\theta|^2 - K',$$ where $K'=\frac{\eta}{\sqrt{\lambda}}\frac{r-1}{r(\sqrt{\lambda}(r-1))^{1/r}}>0$. 

The Cauchy inequality, Lemma \ref{lem:layers_regularization_connection} and the choice of $r$ ensures that \begin{equation}
    \left|\left<\frac{\frac{\partial}{\partial\theta}\left(|h(z;\theta)-l|^2\right)}{1+\sqrt{\lambda}|\theta|^{2 r}},\theta\right>\right|\leq \frac{C(1+|\theta|^{n+2})(1+|y|^2)}{1+\sqrt{\lambda}|\theta|^{2 r}}\leq K' (1+|y|^2 ),
\end{equation} for some $K'>0.$ Now combining these estimates, we get
\begin{equation}
    \left<H(\theta,y),\theta\right>\geq \frac{\eta}{2}|\theta|^2-K'(1+|y|^2),
\end{equation}
 therefore Assumption \ref{asp:dissi} is satisfied with $\Delta=\frac{\eta}{2}$ and $b(y)$ is quadratic in $y$.

We can check that $\gamma=(1-\eta\sqrt{\lambda}+\lambda\eta^2)<1$
in Lemma \ref{obu1} for $\lambda$ small enough so the proof
of Theorem \ref{main} goes through for this choice of $H$.
\end{proof}

Allowing $b$ to be of degree $4$, $\frac{n+2}{2}$ in 
Lemma \ref{triste} could
be decreased to $\frac{n+1}{2}$, as easily seen.

%% appendix sections are then done as normal


%% References



%% New version of the num-names style
\bibliographystyle{elsarticle-num-names}
\bibliography{references.bib}






% \end{thebibliography}

\end{document}